\documentclass[copyright,creativecommons]{eptcs}

\usepackage{graphicx} 
\usepackage{amsmath}
\usepackage{amsthm}
\usepackage{quiver}

\usepackage{xcolor}
\usepackage{placeins}
\definecolor{darkMagenta}{HTML}{DE3163}

\usepackage[all,cmtip]{xy}
\usepackage{amssymb,amsmath,wasysym,url,color}
\usepackage{enumerate,xspace}
\usepackage{graphicx}
\usepackage{proof}
\usepackage{dsfont}
\usepackage{MnSymbol} 
\usepackage{stmaryrd} 

\usepackage{xcolor}

\input diagxy

\usepackage{tikz}

\tikzstyle{strings}=[baseline={([yshift=-.5ex]current bounding box.center)}]

\newcommand{\X}{\ensuremath{\mathbb X}\xspace}
\newcommand{\Y}{\ensuremath{\mathbb Y}\xspace}

\newtheorem{observation}{Remark}[section]
\newtheorem{lemma}[observation]{Lemma}  
\newtheorem{theorem}[observation]{Theorem}
\newtheorem{definition}[observation]{Definition}
\newtheorem{example}[observation]{Example}
\newtheorem{remark}[observation]{Remark}

\newtheorem{proposition}[observation]{Proposition}

\begin{document}

\title{Unitary, Inner Product, and Dagger Categories}


\def\titlerunning{Unitary, inner product, and dagger categories}

\author{
    Robin Cockett
    \institute{
        University of Calgary
            \\
        Calgary, Canada
    }
    \email{robin@ucalgary.ca}
        \and
    Durgesh Kumar
    \institute{
        University of Calgary
            \\
        Calgary, Canada
    }
    \email{durgesh.kumar@ucalgary.ca}
        \and
    Priyaa Varshinee Srinivasan
    \institute{
        Tallinn University of Technology
            \\
        Tallinn, Estonia
    }
    \email{priyaavarshinee@gmail.com}
	\thanks{This work was co-funded by the European
	Union and Estonian Research Council through the Mobilitas 3.0 (MOB3JD1227).}
}

\def\authorrunning{R.\ Cockett, D.\ Kumar, \& P V.\ Srinivasan}

\maketitle

\begin{abstract}

This article provides an alternate characterization of dagger categories, which are central to the study of categorical quantum mechanics, in terms of inner product categories. An inner product category is an ``achiral involutive'' category with an inner product combinator. Inner product categories are, in turn, precisely the same as unitary categories, which are a weaker form of dagger categories. In unitary categories, there is an isomorphism between an object and its dagger, instead of the identity function as in the case of dagger categories. Every unitary category is equipped with a global inner product structure, which allows one to strictify the involutive structure on the unitary category to obtain a dagger category, making unitary categories 2-categorically equivalent to dagger categories.

By regarding the inner product as an abstract metric on an (achiral) involutive category, one can define metric-preserving maps (isometries) in inner product categories, and also develop other notions of special maps --- unitary, Hermitian, positive, and normal maps --- in this setting.
\end{abstract}

\section{Introduction}

A $\dagger$-category (pronounced ``dagger category'') \cite{Selinger} is a category equipped with a contravariant involution: it is a foundational construct in the study of categorical quantum mechanics \cite{Bob-book, Heunen-book}. $\dagger$-categories are defined to have a dagger functor which is strict, in the sense of being the identity on objects.  However, this level of strictness is quite restrictive when relating quantum structures to broader categorical frameworks. For example, $\ast$-autonomous categories \cite{Barr79, Cockett97} and compact closed categories can have necessarily non-strict involutions \cite[Sections 3.1 and 3.2]{Priyaa}. 
In this paper, in order to capture these examples, we follow \cite{Mellies} and reformulate $\dagger$-categories as examples of achiral involutive categories. 

The term {\em involutive} has been used in various ways in the categorical literature. For example, Donald Yau, in his book ``Involutive Category Theory'' \cite{Tau}, studies covariant involutions.  In this regard, he follows a precedent set by Bart Jacobs \cite{Jacobs}. In this paper, however, when we talk about involutive categories, we are considering exclusively contravariant involutions --- also called ``reverse'' involutions by Yau. 

A general theory of contravariant involutions was initiated by Paul-Andr\'e  Mellie\`s in his investigation of ``chirality'' \cite{Mellies}.  The term ``chiral'' comes from chemistry\footnote{In chemistry, a chiral molecule is a molecule with two distinct enantiomers. Even though such enantiomers are mirror images of each other, they have distinct chemical and physical properties (e.g. famously, enantiomers rotate polarized light in opposite directions).} in the context of breaking the behavioural symmetry between the mirror images of a molecule -- here -- between a category and its dual.  In a sequel paper, we shall use the more general {\em chiral\/} setting advocated by Mellie\`s to talk about chiral involutions.  In this paper, however, we consider the {\em achiral} system of 0-cells, 1-cells, and 2-cells which do not break symmetry. In light of this, we will refer to these categories as achiral involutive categories.

Every achiral involutive category has an associated unitary category. A {\em unitary} category is a weak version of a $\dagger$-category \cite{Selinger} in which every object has an isomorphism to its $\dagger$-dual. This associated unitary category is constructed from a general achiral involutive category by collecting its pre-unitary objects. The construction is 2-functorial, and the (2-)category of unitary categories sits as a coreflective subcategory in the 2-category of achiral categories. 

Next, we axiomatize what it means for an achiral involutive category to be equipped with a (global) {\em inner product} structure. Any unitary category can be equivalently viewed as an achiral involutive category with an inner product. This inner product structure can be used to strictify the involutive structure to obtain a $\dagger$-category (in the sense of \cite{Selinger}), providing a three-way equivalence between the 2-categories of unitary categories, inner product categories, and $\dagger$-categories. Moreover, by regarding the inner product as providing an abstract metric on the category, one can define isometries to be metric-preserving maps. Unitary maps are then isometries which are invertible.  Similarly, Hermitian maps, positive maps, and normal maps can all be naturally described in terms of the inner product structure. 

Chiral involutive categories are not unitary, hence they do not have an inner product. However, this setting naturally leads to $\dagger$-linearly distributive categories \cite{Priyaa} where the par is given by the chiral dual. Notably, models of $\dagger$-linearly distributive categories ($\dagger$-LDCs) are used to study infinite-dimensional quantum processes. We plan to explain this in detail in a sequel.

\section{Achiral involutions}

A category has an {\bf (achiral) involution} if it has a contravariant endofunctor $(\_)^\dagger: \mathbb{X}^{\rm op} \to \mathbb{X}$ called the {\bf dagger} and a natural isomorphism 
$\iota: X \to X^{\dagger\dagger}$ called the {\bf involutor} such that the following diagram commutes:- 
\[\begin{tikzcd}
	{X^\dagger} \\
	{X^{\dagger \dagger \dagger}} && {X^\dagger}
	\arrow["{\iota_{X^\dagger}}"', from=1-1, to=2-1]
	\arrow[shift left, no head, from=1-1, to=2-3]
	\arrow[no head, from=1-1, to=2-3]
	\arrow["{(\iota_X)^\dagger}"', from=2-1, to=2-3]
\end{tikzcd}\]
We call a category with an (achiral) involution an {\bf involutive category} and write it as a triple $\X = (\X, (-)^\dagger, \iota)$.  We shall call an involutive category whose involution is {\bf strict} (so that $X^\dagger = X$ and $\iota$ is the identity map) a {\bf strict involutive category} and use this term interchangeably with $\dagger$-category. 

\begin{example}~ 
\label{Example: involutive}
{\em
\begin{enumerate}[(1)]
\item All $\dagger$-categories as defined in \cite{Selinger} are strict involutive categories and there are many examples of these: 
\begin{enumerate}
\item Every groupoid (and, indeed, more generally, every inverse category\footnote{A category in which for every morphism $f$, there exists a unique map $g$ such that $fgf=f$ and $gfg=g$.}) is a $\dagger$-category with $x^\dagger = x^{-1}$  (respectively $x^\dagger = x^{(-1)}$). 
\item The category of relations with $R^\dagger = R^\circ = \{ (b,a) | (a,b) \in R \}$ being the transpose relation. 

\item The category of matrices over any rig (or ring), $R$, which has an involution is a $\dagger$-category with respect to the conjugate transpose.  A rig has a conjugation, $\overline{(\_)}$, when $\overline{1} = 1$, $\overline{ab} = \overline{b}~\overline{a}$, $\overline{0} = 0$, and $\overline{a + b} = \overline{a}+\overline{b}$.  When the rig is commutative, there is a trivial conjugation given by the identity.   One then obtains a strict involution by modifying the transpose:
\[ \infer[(\_)^\dagger]{(\overline{r_{i,j}})_{j<m,i<n}: m \to n}{( r_{i,j})_{i<n,j<m}: n \to m} \]
This gives the usual dagger (conjugate transpose) when $R = \mathbb{C}$ (where $\mathbb{C}$ is the complex numbers with the usual conjugation).
\end{enumerate}

\item Every $*$-autonomous category and compact closed category is an involutive category with the involution being the dual functor.
\item An elementary example of an involutive category is the category of symmetric reflexive graphs, ${\sf SRGraph}$.  An object is a set $X$ equipped with a symmetric reflexive relation, $\_\smile\_$ (with $x \smile x$ and, if $x \smile y$, then $y \smile x$). A morphism is an isomorphism/bijection between the underlying sets which preserves the relation (that is $\alpha: X \to Y$, which has an inverse, $\alpha^{-1}$ as a set map,  and is such that $x \smile x'$ implies $\alpha(x) \smile \alpha(x')$).

The involution $(\_)^\dagger: {\sf SRGraph}^{\rm op} \to {\sf SRGraph}$ sends $\alpha: (X,\smile_X) \to (Y,\smile_Y)$ to $\alpha^\dagger = \alpha^{-1}: (Y,\smile_Y^\dagger) \to (X,\smile_X^\dagger)$, where $x \smile_X^\dagger x'$ iff $x \not{\smile} x'$ or $x=x'$.  Notice that this is not a $\dagger$-category (as $(X,\smile_X) \not{=} (X,\smile_X)^\dagger$ - except for the empty and one element objects) but $\iota: X \to X^{\dagger\dagger}$ is the identity.
\item The integers with respect to the usual ordering is an involutive category with respect to negation -- notice this is not a dagger category as $n^\dagger:= -n \not{=} n$, unless $n=0$. More generally, any finite ordinal is an involutive category by reversing the order.
\end{enumerate}
}
\end{example}

\begin{remark} {\em The axiom $\iota_{X^\dagger} (\iota_X)^\dagger = 1_{X^\dagger}$ for achiral involution gives an adjoint equivalence:
\[ \xymatrix{ \mathbb{X} \ar@/^1pc/[rr]^{(\_)^\dagger} \ar@{}[rr]|\bot && \mathbb{X}^{\rm op} 
            \ar@/^1pc/[ll]^{(\_)^\dagger} } \]
with the unit and the counit being $\iota: X \to X^{\dagger\dagger}$.  It should be noted that most of the theory we develop does not require this identity.  As shall be seen, the identity is needed to make $\iota$ an achiral (natural) transformation.}
\end{remark}

\begin{definition}
\label{defn: involutive-functor}
    If $(\mathbb{X},(\_)^\dagger,\iota)$ and $( \mathbb{Y},(\_)^\ddagger,\iota')$ are categories with involution  then an {\bf (achiral) involutive functor} written as, 
$ (F,\gamma): (\mathbb{X},(\_)^\dagger,\iota) \to (\mathbb{Y},(\_)^\ddagger,\iota'), $ consists of a functor $F: \mathbb{X} \to \mathbb{Y}$ and a natural isomorphism  $\gamma_X: F(X^{\dagger}) \to F(X)^{\ddagger}$ (of the contravariant functors)  called the {\bf preservator} making the following diagram commute:
\[ \xymatrix{  F(X) \ar[d]_{\iota'} \ar[rr]^{F(\iota)} & &
              F(X^{\dagger\dagger}) \ar[d]^{\gamma_X} \\
              F(X)^{\ddagger\ddagger} \ar[rr]_{\gamma_X^\ddagger}& & F(X^\dagger)^{\ddagger} 
             } \]
Involutive functors compose with the composition defined as  $(F,\gamma)(G,\gamma') := (FG,G(\gamma)\gamma')$. 
\end{definition}

\begin{remark} 
{\em ~
\begin{enumerate}[(1)]
\item Here we are asking more coherence for functors than is usual.  Often one asks just for the existence of $\gamma_X: F(X^\dagger) \to F(X)^\dagger$ with no coherence requirement governing the interaction with $\iota$.  
\item For $\dagger$-categories one  asks that $F$ ``preserves'' the dagger, that is $F(X^\dagger)$ and $F(X)^\dagger$ are literally the same object.  This, and the fact that the involutor is the identity, collapses the coherence square above to the identity. This ensures that this coherence requirement is automatically satisfied in $\dagger$-categories. 
\item  For covariant involutions one can prove that the preservator must be an isomorphism from the coherence requirement \cite{Jacobs,Tau}.  However this does not appear to be true for the contravariant involution: so we have {\em required\/} that the preservator be an isomorphism.
\end{enumerate}
}
\end{remark}

\begin{definition}
If $(F,\gamma), (G,\gamma'): \mathbb{X} \to \mathbb{Y}$ are involutive functors then an {\bf (achiral) transformation} $\alpha: (F,\gamma) \to (G,\gamma')$ is a natural transformation $\alpha: F \to G$ which in addition satisfies:
\[ \xymatrix{ F(X) \ar[dd]_{F(\iota)}\ar[rr]^{\iota'_{F(X)}} \ar@{..>}[dr]_{\alpha_X} & & 
       F(X)^{\ddagger\ddagger} \ar@{..>}[dr]^{(\alpha_{X})^{\ddagger\ddagger}} \ar[dd]^<<<<<<<<<{\gamma_X^\ddagger}|\hole  \\ 
       & G(X) \ar[rr]_<<<<<<<<<<{\iota_{G(X)}'} \ar[dd]^<<<<<<<{G(\iota)}  && G(X)^{\ddagger\ddagger} \ar[dd]^{\gamma_{X^\dagger}} \\
       F(X^{\dagger\dagger}) \ar@{..>}[dr]_{\alpha_{X^{\dagger\dagger}}} \ar[rr]_>>>>>>>>{\gamma_{X^\dagger}}|\hole && F(X^\dagger)^\ddagger \\
       & G(X^{\dagger\dagger}) \ar[rr]_{\gamma'_{X^\dagger}} && G(X^\dagger)^\ddagger \ar@{..>}[lu]^{\alpha_{X^\dagger}^\ddagger} } \]
\end{definition}

Notice that the bottom right morphism is going in the ``wrong direction'':  we call this map the {\bf tether}.  The requirement that the squares in which the tether occurs must commute (the tethering requirement) has a significant consequence: 

\begin{lemma}
    The tethering requirement forces $\alpha$ to be a natural isomorphism.
\end{lemma}

\begin{proof}
To see this, first note that a morphism $f$ is a retraction (respectively a section) if and only if any conjugate, $\beta f \gamma^{-1}$ -- with $\beta$ and $\gamma$ invertible -- is a retraction (respectively a section).

Notice that the bottom square of the cube above forces $\alpha_{X^{\dagger\dagger}}$ to be a section: whence, as $\alpha_X = F(\iota) \alpha_{X^{\dagger\dagger}}G(\iota)^{-1}$, $\alpha_X$ is a section.  Now notice that the right-hand square forces $\alpha_{X^\dagger}^\ddagger$ to be a retraction, whence, as $\alpha_X$ is a conjugate of this, we also have that $\alpha_X$ is a retraction and hence an isomorphism.
\end{proof}
More generally, as we shall see when we discuss chiral transformations in a sequel paper, there is another less drastic way to explain ``tethering'' without requiring all transformations to be isomorphisms.  Furthermore, as we shall see, this has interesting structural ramifications.

\begin{remark}~
{\em
\begin{enumerate}[(1)]
\item
In the case of $\dagger$-categories both $\iota$ and $\gamma$ are identities and this 
forces $\alpha^{-1} = \alpha^\dagger$.  Thus, $\alpha$, in this case, is a unitary isomorphism \cite[Definition 2.3]{Selinger}.
\item As the tether is an isomorphism, its direction can be corrected.  Indeed, for a chiral transformation \cite{Mellies}, there are actually two natural transformations going in opposite directions: 
in the achiral situation this pair becomes $(\alpha,\alpha^{-1})$.
\end{enumerate} }
\end{remark}


It is clear that these transformations compose by just stacking these squares.  It is also clear that whiskering on the left gives a transformation:
\[ \xymatrix{ \mathbb{X} \ar[rr]^V &&  \mathbb{Y} \ar@/^1pc/[rr]^{F} \ar@/_1pc/[rr]_{G} \ar@{}[rr]|{\Downarrow \alpha} && \mathbb{Z}  \ar[rr]^W && \mathbb{W} } \]
 Whiskering on the right is more complicated  (see commuting diagram in Figure \ref{fig:Whisker} in Appendix \ref{Sec: App Achiral involutions}). This leads to the claim that:   
\begin{proposition}
Involutive categories, their functors, and their achiral transformations organize themselves into a 2-category, ${\sf aChiral}$ (in which all 2-cells are isomorphisms).
\end{proposition}

There is a ready-to-hand example of such a transformation of involutive functors:

\begin{lemma}
For any achiral involutive category the involutor itself is an (achiral) transformation.
\end{lemma}
The proof is given in Appendix \ref{Sec: App Achiral involutions}.



\section{Unitary categories}

The notion of a unitary category was introduced in \cite{Priyaa} in the context of dagger isomix categories. Here, we extract the essence of the notion to involutive categories. 

An involutive category, $\mathbb{X}$ is a {\bf unitary  category} \cite[Definition 4.1]{Priyaa} in case for every object $X \in \mathbb{X}$ there is an isomorphism 
$\varphi_X: X \to X^\dagger$, called the {\bf unitary structure map} of $X$ (not assumed to be natural) such that: 
\begin{enumerate}[{\bf [U.1]}]
\item$\varphi_{X^\dagger} = (\varphi_X^{-1})^\dagger$ where $\varphi_{X^\dagger}: X^\dagger \to X^{\dagger\dagger}$;
\item $\varphi_X  \varphi_{X^\dagger} = \iota_X$ where $\iota_X: X \to X^{\dagger \dagger}$.
\end{enumerate}
 A unitary category is thus a quadruple consisting of the involutive category and a unitary structure --- $\X = (\X, (-)^\dagger, \iota, \varphi)$.
Unitary structure maps satisfy a number of identities:
\begin{lemma} \label{unitary-identities}
The unitary structure maps satisfy:
\[ (i)~~\varphi_X^\dagger = \varphi_{X^\dagger}^{-1}; \quad\quad
(ii)~~\varphi_X = \iota (\varphi_X^\dagger); \quad\quad (iii)~~\varphi_X^{-1}\iota =  (\varphi_X^{-1})^\dagger; \quad\quad (iv)~~\iota (\varphi_{X^\dagger})^\dagger \iota^{-1} = \varphi_X^{-1}. \]
\end{lemma}

\begin{example}{\em ~
\begin{enumerate}[(1)]
\item Any $\dagger$-category has a unitary structure given by the identity map.
\item The standard example of a unitary $\dagger$-category is the category of finite dimensional vector spaces:  every finite dimensional space is isomorphic to its dual, albeit in an unnatural way (depends on the choice of a basis): these isomorphisms give a unitary structure \cite[Section 4.4.1]{Priyaa}.
\item
The integers under the usual ordering with the negation as involution is a simple example of a non-unitary category. 
\end{enumerate}
} \end{example}

\begin{definition}
      An involutive functor $(F, \gamma): (\mathbb{X}, (\_)^{\dagger}, \iota) \to (\mathbb{Y}, (\_)^{\ddagger}, \iota')$, between unitary categories is a {\bf unitary functor} if for every $X \in \mathbb{X}$, the following diagram commutes:
\[\xymatrix{ F(X) \ar[d]_{F(\varphi_X)} \ar[drr]^{\varphi_{F(X)}} 
   & &  \\  F(X^\dagger) \ar[rr]_{\gamma_X} & &F(X)^\dagger }\]
\end{definition}

\begin{remark}{\em ~
 \begin{enumerate}[(1)]
    \item For a unitary functor $(F,\gamma)$, the unitary structure determines the preservator as $\gamma_X= F(\varphi_X^{-1}) \varphi_{F(X)}$ and this forces the preservator to be an isomorphism (recall that, in general, we had to require this).
    \item A unitary functor between $\dagger$-categories necessarily preserves the involution on the nose as the preservator $\gamma: F(X^\dagger) \to F(X)^\dagger$ is the identity as $F(\varphi_X)\gamma = \varphi_{F(X)}$ and the unitary structure maps are all identities forcing $\gamma$ to be the identity.
  \end{enumerate}}
\end{remark}

\begin{lemma}~ \label{Lemma 5.6}
\begin{enumerate}[(i)]
    \item Unitary functors between $\dagger$-categories are necessarily strict involutive functors which preserve the dagger on the nose, $F(X^\dagger) = F(X)^\dagger$;
    \item For unitary functors, the underlying functor determines the preservator.
\end{enumerate}
\end{lemma}

\begin{proof} It remains to prove the latter statement. Suppose $(F,\gamma)$ and $(F,\gamma')$ are unitary functors with the same underlying functor, $F$, then, as $\varphi_{F(X)} = F(\varphi_X) \gamma$, we have $\gamma= F(\varphi_X^{-1}) \varphi_{F(X)}$.
\end{proof}

\begin{lemma}
    Unitary categories with unitary functors and achiral transformations form a 2-category, ${\sf UNITARY}$.
\end{lemma}
\begin{proof}
It is clear that the identity functor (with the identity natural transformation) is a unitary functor. We now show that the composition of two unitary functors is unitary:
\[
    G(F(\varphi_X)) G(\gamma_X) \gamma'_X  =  G(\phi_{F(X)}) \gamma'_{F(X)} = \psi_{G(F(X))}
\]
Unitality and associativity of composition follow from that of involutive functors. 
\end{proof} 

 We now construct a 2-coreflection between unitary categories and involutive categories.  It is clear that the unitary categories form a sub-2-category of involutive categories, ${\cal I}$ $:  {\sf UNITARY} \to {\sf aChiral}$.  We extract a unitary category from an arbitrary involutive category adapting the unitary construction \cite[Definition 4.12]{Priyaa}.

Given an involutive category $\X$ we form ${\sf Unitary}(\X)$ as follows:

\begin{description}
    \item[Objects:] Pairs $(X, h: X \to X^\dagger)$, where $X$ is an object of $\mathbb{X}$ and $h$ is an isomorphism such that $h (h^{-1})^\dagger = \iota$. Each such pair is called a pre-unitary object of $\X$.
    \item[Maps:]  A map $f: (X, h: X \to X^\dagger) \to (Y, h': Y \to Y^\dagger)$ is any map $f: X \to Y$ of $\X$.
    \item[Involution:]  On objects is $(X,h)^\dagger := (X,h^{-1})^\dagger$ and, on maps and for $\iota$, as in $\X$.
\end{description}

We revisit some of the examples of involutive categories (from Example \ref{Example: involutive}) to illustrate the pre-unitary objects in these categories:
\begin{example}~ {\em 
    \begin{enumerate}[(1)]
        \item Let $\X$ be any $\dagger$-category. The pre-unitary objects of $\X$ are $(X, h)$ where $h$  is  a Hermitian isomorphism, $(f=f^\dagger)$ and $\mathbb{X}$ embeds as a full subcategory into $\sf Unitary(\mathbb{X)}$ by sending an object $X$ to $(X, id_X)$.
        \item For the category ${\sf SRGraph}$, recall that the dagger of a map is just its inverse function. For a pre-unitary object we need to have an isomorphism between $h: X \to X^\dagger$ which must preserves the relation and, furthermore, is such that $h (h^{-1})^\dagger = 1$. In ${\sf SRGraph}$ the latter identity means that the underlying set function has $hh =1$. Suppose $h(a) = b$ with $a \not= b$ then if $a \sim b $ then $h(a) \not\sim h(b)$ and conversely if $a \not\sim b$ then $h(a) \sim h(b)$ so that there cannot be any distinct points which are swapped. This forces $h$ to be the identity map; however, a similar argument shows that $h$ cannot fix two distinct points in this case either. Thus, the cardinality of the underlying set must be at most $1$. So the pre-unitary objects of ${\sf SRGraph}$ are precisely the empty relation on the null set and the unique reflexive relation on a singleton set.
        \item In the posetal category of integers, the {\sf Unitary} construction produces the one-object discrete category $0$.  Note that for the subcategory of non-zero integers there are no pre-unitary objects.
        \end{enumerate} }
\end{example}

\medskip
Given $F: \X \to \Y$ an achiral functor, then we may transfer it onto the unitary objects by taking $h: X \to X^\dagger$ to $F(h)\gamma: F(X) \to F(X)^\dagger$.  We must check that  $F(h)\gamma: F(X) \to F(X)^\dagger$ is a pre-unitary object:
\begin{eqnarray*}
 F(h)\gamma ((F(h)\gamma)^{-1})^\dagger & = &  F(h) \gamma (\gamma^{-1} F(h)^{-1})^\dagger = F(h) \gamma  (F(h)^{-1})^\dagger (\gamma^{-1})^\dagger \\
 & = & F(hh^{-1\dagger}) \gamma (\gamma^{-1})^\dagger = F(\iota) \gamma (\gamma^{-1})^\dagger = \iota.  
\end{eqnarray*}

That $F$ on pre-unitary objects is a unitary functor follows from the definition of $F$ on objects.  Thus, ${\sf Unitary}: {\sf aChiral} \to {\sf UNITARY}$ is a 2-functor:

\begin{proposition} There is a 2-adjunction: 

  \[ {\sf UNITARY}(\mathbb{U},{\sf Unitary(\X)) \simeq {\sf aChiral}}({\cal I}(\mathbb{U}),\X) \]
  where $\mathbb{U}$ is a unitary category and $\mathbb{X}$ is an involutive category with ${\sf Unitary}({\cal I}(\mathbb{U})) \simeq \mathbb{U}$ making this a 2-coreflection. 
\end{proposition}

\section{Inner product categories}

Our next aim is to show that every unitary category comes equipped with an inner product structure and, indeed, to show that this structure actually characterizes unitary categories.  To this end, we start with:

\begin{definition}
   An {\bf inner product category} is an involutive category which has an {\bf inner product combinator}:
   \[ \infer{\langle f|g \rangle: A \to B^\dagger}{f: A \to X & g: B \to X} \]
    satisfying axioms {\bf [IP.1]}--{\bf [IP.4]} listed below.

    \begin{enumerate}[{\bf [{IP}.1]}]
        \item {\bf Pre-composition:} $\langle hf | kg \rangle = h \langle f | g \rangle k ^\dagger$.
        \item {\bf Dagger and duals:} (a) $\iota \langle f^\dagger | g^\dagger \rangle^\iota = \langle f | g \rangle $ and (b) $\iota \langle f | g \rangle^\dagger = \langle f | g \rangle$.
         \item Given $hg = 1$, $\langle f | g \rangle \langle h | k \rangle^\iota = fk$
        \item Given $gh = 1$, $\langle f | g \rangle^\iota \langle h | k \rangle = (kf)^\dagger$,
    \end{enumerate} 
    where the map $\langle h | k \rangle^\iota: A^\dagger \to B$ is the {\bf dual inner product} of $h: X \to A$ and $k: X \to B$   defined as follows:    
        \[ \text{\bf[dual-IP] } \quad \quad  
        \begin{tikzcd}
        	{A^\dagger} & {A^{\dagger \dagger \dagger}} \\
        	B & {B^{\dagger \dagger}}
        	\arrow["{\iota_{A^\dagger}}", from=1-1, to=1-2]
        	\arrow["{\langle h | k \rangle^\iota}"', from=1-1, to=2-1]
        	\arrow["{:=}"{description}, draw=none, from=1-1, to=2-2]
        	\arrow["{\langle k^\dagger | h^\dagger \rangle^\dagger}", from=1-2, to=2-2]
        	\arrow["{\iota_B^{-1}}", from=2-2, to=2-1]
        \end{tikzcd} \]
\end{definition}
For the types of maps in {\bf [{IP}.1] - [{IP}.4]}, see, Definition \ref{Defn: App inner product} in Appendix \ref{Sec: App inner product}.

\begin{lemma}
\label{Lem:double-dagger-inner-product}
In an inner product category, the following are equivalent: 
\begin{enumerate}[(a)]
    \item Either {\bf [IP.2](a)} or {\bf [IP.2](b)} holds, and {\bf [IP.2](c)} $\langle f^{\dagger \dagger} | g^{\dagger \dagger} \rangle = \iota_A^{-1} \langle f | g \rangle \iota_{B^\dagger}$ holds;
    \item Both {\bf [IP.2](a)} and {\bf [IP.2](b)} hold.
\end{enumerate}
\end{lemma}
For proof of the above Lemma and subsequent results in this section, see Appendix \ref{Sec: App inner product}.

\begin{lemma} \label{inner_identities} 
The following equations hold in an inner product category:
\begin{enumerate}[(i)]
\item If $h,k: B \to A$ and $f,g: A \to X$ then $\langle hf|kg \rangle^\iota = f^\dagger \langle h|k \rangle^\iota g$
\item $\langle g^\dagger|f^\dagger\rangle = (\langle f|g \rangle^\iota)^\dagger$
\item $\langle f|g \rangle^\dagger = \langle g^\dagger | f^\dagger \rangle^\iota$
\end{enumerate}
\end{lemma}

\begin{definition}
\label{Defn: inner-product-functor}
    An involutive functor $(F, \gamma): \mathbb{X} \to \mathbb{Y}$ between inner product categories {\bf preserves} the inner product if for any $f, g: A \to X \in \mathbb{X}$, $\langle F(f)| F(g) \rangle = F (\langle f|g \rangle) \gamma_A$.
\end{definition}


Given any unitary category, for any pair of maps $f: A \to X$ and $g: B \to X$ we may define an {\bf inner product} $\langle f | g \rangle := f \varphi_X g^\dagger: A \to B^\dagger$.

\begin{lemma} \label{inner-to-unitary} 
    Every unitary category is an inner product category with the inner product combinator defined as follows. 
    Given $f: A \to X, g: B \to X$, their inner product is defined  to be: 
    \[\langle f|g \rangle_X := f \varphi_X g^\dagger \]
\end{lemma}

\begin{proof}  We begin computing the dual inner product. Given $h: X \to A$ and $k: X \to B$,
\begin{eqnarray*}
    \langle h|k \rangle^\iota & := &  \iota_{A^\dagger} \langle k^\dagger|f^\dagger \rangle^\dagger \iota^{-1}_B = \iota_{A^\dagger} (k^\dagger \varphi_{X^\dagger} h^{\dagger\dagger})^\dagger \iota^{-1}_B \\
    & = & \iota_{X^\dagger} h^{\dagger\dagger\dagger} (\varphi_{X^\dagger})^\dagger k^{\dagger\dagger} \iota^{-1}_B = f^\dagger \iota_{X^\dagger} (\varphi_{X^\dagger})^\dagger \iota^{-1}_X k \\
    & = & h^\dagger \varphi_X^{-1} k
\end{eqnarray*}
where we use in the last line (see also Lemma \ref{unitary-identities}):
$\iota (\varphi_{X^\dagger})^\dagger \iota^{-1}= \iota (\varphi_X^{-1})^{\dagger\dagger} \iota^{-1} = \varphi_X^{-1}$.

We now prove the identities: Given $f: A \to X$ and $B \to X$, {\bf [{IP}.2]} is immediate. For,
\begin{enumerate}[{\bf [{IP}.2]}]
\item (a) $\iota \langle f^\dagger | g^\dagger \rangle^\iota 
= \iota f^{\dagger\dagger} \varphi_{X^\dagger}^{-1} g^\dagger = f \iota \varphi_{X^\dagger}^{-1} g^\dagger = f \varphi_X g^\dagger = \langle f|g \rangle$ \\
(b)~ $\iota \langle f|g \rangle^\dagger = \iota (f \varphi_X g^\dagger )^\dagger = \iota g^{\dagger\dagger} \varphi_X^\dagger f^\dagger = g \iota \varphi_X^\dagger f^\dagger = g \varphi_X f^\dagger = \langle g|f \rangle$
\item If $hg = 1$ then, 
$ \langle f|g \rangle \langle h| k \rangle^\iota = f \varphi_X g^\dagger h^\dagger \varphi_X^{-1} k
       = f \iota_X (hg)^\dagger \iota_X^{-1} k = fk $
\item If $gh=1$ then, $ \langle f|g \rangle^\iota \langle h| k \rangle 
       = f^\dagger \varphi_X^{-1} g h \varphi_X k^\dagger 
       = f^\dagger  k^\dagger = (kf)^\dagger $ \qedhere
\end{enumerate}
\end{proof}


\begin{proposition} \label{inner_2_unitary}
   An inner product category is a unitary category where the unitary structure is given by: $ \varphi_A := \langle 1_A|1_A \rangle: A \to A^\dagger.$
\end{proposition}

\begin{proof} 
We show that $\varphi_A := \langle 1_A|1_A \rangle: A \to A^\dagger$ is the unitary structure map, that is {\bf [U.1]} and {\bf [U.2]} hold. 

First $\varphi_A: A \to A^\dagger$ is an isomorphism: its inverse is $\varphi_A^{-1} := \langle 1_A|1_A \rangle^\iota$. By {\bf [IP.3]} we have that $\varphi_A^{-1} \varphi = 1_A$, and by {\bf [IP.4]} we have that $\varphi_A \varphi_A^{-1} = 1_{A^\dagger}$.
\begin{description}
    \item[{\bf [U.1]}] $\varphi_{A^\dagger} = (\varphi_A^{-1})^\dagger$
\[  \varphi_{A^\dagger} :=  \langle 1_{A^\dagger} | 1_{A^\dagger} \rangle  =  \langle 1_{A}^\dagger | 1_{A}^\dagger \rangle           \stackrel{Lem. \ref{inner_identities}(ii)}{=} \langle 1_A | 1_A \rangle^{\iota^\dagger} =: (\varphi_A^{-1})^\dagger \]
    
 \item[{\bf [U.2]}] $\varphi_A \varphi_{A^\dagger} = \iota$
        \[ \varphi_A := \langle 1_A|1_A\rangle \stackrel{\mbox{{\bf [IP.2]}(a)}}{=} \iota \langle 1_A^\dagger| 1_A^\dagger \rangle^\iota = \iota \langle 1_{A^\dagger}| 1_{A^\dagger} \rangle^\iota
=: \iota \varphi_{A^\dagger}^{-1} \] 
\end{description} 

Finally, we note that, almost immediately, the unitary structure induced by an inner product structure induces the original inner product structure and, conversely, the inner product structure of unitary structure induces the original unitary structure; thus, the structures are in bijective correspondence. 
\end{proof}

Notice Proposition \ref{inner_2_unitary} does not use {\bf [IP.2]}(b).  This means that demanding the identity is redundant. This means that one should be able to prove {\bf [IP.2]}(b) from {\bf [IP.2]}(a) as shown in the below Lemma. 

The proof uses {\bf [IP.1]} and {\bf [IP.4]} to reduce the problem to an identity of unitary structure maps and then uses the fact that the unitary structure maps are isomorphisms. 

\begin{theorem}
    A unitary category is precisely the same as an inner product category. 
\end{theorem}
\begin{proof}
    Follows immediately from Propositions \ref{inner-to-unitary} and \ref{inner_2_unitary}. 
\end{proof}

Thus, an inner product can be used to replace unitary structure and vice versa. Furthermore, as we now note, the equivalence of inner product and unitary structure also works at the level of functors:

\begin{lemma} \label{Lemma 6.6}
    Let $\mathbb{X}$ and $\mathbb{Y}$ be inner product categories. Then $(F, \gamma): \mathbb{X} \to \mathbb{Y}$, an involutive functor is inner product preserving if and only if it is a unitary functor.
\end{lemma}

Global inner products are not a new idea: Vicary \cite{Vicary}, for example, defines the inner product of two parallel maps $f,g: A \to X$ in a dagger category as $\langle f | g \rangle = f g^\dagger$. Note that our definition of inner product reduces to Vicary's for $\dagger$-categories and parallel maps.  However, our definition does encompass non-parallel maps as well. Thus, every dagger category is certainly an example of an inner product category. 

Viewing a groupoid as a $\dagger$-category, the inner product of two morphisms $x: X \to Z, y: Y \to Z$ will be $xy^{-1}$. In the category of relations the inner product of $R: X \to Z, S: Y \to Z$ will be $RS^\dagger: X \to Y$, where $S^\dagger$ is the transpose relation.

\subsection{Inner adjoints}

At this stage, we know that a unitary category can equivalently be defined as a category with an inner product.  However, a category with an inner product supports the notion of an ``inner adjoint''.  
\begin{definition}
    In an inner product category, given $f:A \to B$ the (inner) {\bf adjoint} of $f$ is a map $f^*: B \to A$ such that 
    \[ \langle x f | y \rangle = \langle x | y f^* \rangle\]
\end{definition}

Inner adjoint provide the generalization of ``adjoint" in the sense of Hilbert spaces.  They also provide a clean way of showing that a unitary category is equivalent to a $\dagger$-category.

\begin{lemma}
\label{Lem:Adjoint-dagger}
In an inner product category, the adjoint of a map $f: A \to B$ exists, is unique, and is given by 
\begin{equation}
\label{eqn: adjoint}
    f^*:= \varphi_B f^\dagger \varphi_A^{-1}
\end{equation} 
which satisfies:
\[ 
(i)~~\langle x f | y \rangle = \langle x | y f^* \rangle \quad \quad \quad  (ii)~~ \langle x f^* | y \rangle = \langle x | y f \rangle \quad \quad  (iii)~~ f = (f^*)^*
\quad \quad  
  (iv)~~ 1_A^* =1_A \text{ and }(fg)^* = g^*f^*     \]
\end{lemma}

\begin{proof}
Let $f: A \to B$, setting $x=1_A$ and $y=1_B$ gives: 
\[ \infer={
    \infer={
     \infer={
      \infer={f^* = \varphi_B f^\dagger \varphi_A^{-1}
             }{\iota \varphi_B^\dagger f^\dagger (\iota \varphi_A^\dagger)^{-1} 
                   = f^*}
           }{\iota (\varphi_A^{-1} f \varphi_B)^\dagger \iota^{-1} 
                   = \iota (f^*)^{\dagger\dagger} \iota^{-1}}
           }{\varphi_A^{-1} f \varphi_B = (f^*)^\dagger}
        }{\infer{f \varphi_B = \varphi_A (f^*)^\dagger}{\langle f|1 \rangle = \langle 1| f^* \rangle}}
 \]
 This shows that $f$  determines $f^*$ uniquely.  Similarly, $f^*$ determines $f$ as 
 \[ f  = \varphi_A (f^*)^\dagger \varphi_B^{-1} \]

 We may now use these for establishing the Lemma:
 \begin{enumerate}[{\em (i)}]
\item This is by definition.

\item  This means, when $f: A \to B$, that we have both $\langle x | y f^\ddagger \rangle = \langle x f|y \rangle$ and $\langle x f^*| y \rangle = \langle x |y f\rangle$ where we define $f^* := \varphi_B f^\dagger \varphi_A^{-1}: B \to A$.  

\item This, in particular, does means $f^{**} = f$. We may also calculate this out directly:
\begin{eqnarray*}
    f^{**} & := & \varphi_A (f^*)^\dagger \varphi_B^{-1} 
     :=  \varphi_A (\varphi_B f^\dagger \varphi_A^{-1})^\dagger \varphi_B^{-1}  =  \varphi_A (\varphi_A^{-1})^\dagger f^{\dagger\dagger} \varphi_B^\dagger \varphi_B^{-1} = \iota f^{\dagger\dagger} \iota^{-1} = f.
\end{eqnarray*}
\item
Now this $(\_)^*$ is clearly a contravariant functor as it is easily seen that it preserves identities and for composition we have:
\[ \langle x|y(fg)^* \rangle = \langle xfg | y \rangle 
      = \langle xf | y g^* \rangle = \langle x | y g^* f^* \rangle  \text{ and so } (fg)^* = g^* f^* \qedhere \]
\end{enumerate}
\end{proof}
  
  Thus, taking the adjoint of a map is 
  a contravariant, stationary-on-objects functor with
  $f^{**} = f$; with this structure, we have turned our inner product category $(\mathbb{X},(\_)^\dagger)$ into a dagger category 
  $(\mathbb{X},(\_)^*)$: we refer to this latter structure as the adjoint involution.

  \begin{proposition}\label{induced dagger category}
      For every inner product category $\mathbb{X}$, the adjoint-involution $(\mathbb{X},(\_)^*)$ gives a dagger category. 
  \end{proposition}
  \begin{proof}
     It follows from Lemma \ref{Lem:Adjoint-dagger} $(iii)$ that $f^{**} = f$, and from $(iv)$ that $1_A^* = 1_{A^*}$ and $(fg)^* = g^* f^*$.  Thus, $(\_)^*$ is a contravariant involution on $\X$.
  \end{proof}
  
  Furthermore, the identity functor now becomes a morphism of unitary categories: 
  \[ V_{\mathbb{X}}: := ({\sf Id},\varphi^{-1}) : (\mathbb{X},(\_)^\dagger) \to (\mathbb{X},(\_)^*)\]
  with the preservator $\varphi_X^{-1}: X^\dagger \to X$ (recall that $X^* = X$). This is clearly a natural transformation as for any $f: X \to Y \in \mathbb{X}$, 
\begin{eqnarray*}
    \varphi_Y^{-1} f^* \stackrel{\text{Eqn }(\ref{eqn: adjoint})
}{:=} \varphi_Y^{-1} \varphi_Y f^\dagger \varphi_X^{-1}= f^\dagger \varphi_X^{-1}
\end{eqnarray*}
And it satisfies the required coherence condition as: 
\begin{eqnarray*}
    \iota_X \varphi_{X^{\dagger}}^{-1}  ~=~ \varphi_X ~=~ \varphi_X (\varphi_X^{-1})^{\dagger} (\varphi_X)^{\dagger}  ~=~ \varphi_X (\varphi_X^{-1})^{\dagger} ((\varphi_X)^{\dagger})^{-1} ~=~ \varphi_X (\varphi_X^{-1})^{\dagger} \varphi_{X^\dagger}^{-1} ~=~ (\varphi_X^{-1})^{\ddagger}
\end{eqnarray*}

Now note that as every $\dagger$-category has an inner product structure we may define a new inner product based on the adjoint, $\langle f|g \rangle^*:= f g^*$. The functor $({\sf Id},\varphi^{-1})$ then preserves the inner product structure as 
\[ \langle f|g \rangle \varphi^{-1} = f \varphi g^\dagger \varphi^{-1} =: f g^* =: \langle f|g \rangle^* \] 
Thus by Lemma \ref{Lemma 6.6} it is a unitary functor. The functor is clearly a (unitary) equivalence. We also have an equivalence between the category of inner product (or unitary) categories and unitary functors, and the category of dagger categories and unitary functors. We define this functor by sending an inner product category to its induced dagger category and we send a unitary functor $(F, \gamma)$ to $(F, id)$. This functor is clearly bijective on objects and is full. Faithfulness of this functor follows from equation (\ref{eqn: adjoint}) for conjugation.
So we have the following proposition:

\begin{proposition}
    There is an equivalence between the category ${\sf Unitary}$ of inner product (or unitary) categories, unitary functors and transformations, and the full sub-2-category {\em $ \dagger$}-${\sf Cat}$ of dagger categories.
\end{proposition}

We define the equivalence by sending a unitary category to its $\dagger$-category and transport the 1-cells and 2-cells by sandwiching $\alpha_\X \mapsto V^{-1}_\X \alpha_\X V_\X$.


\subsection{Isometries and special maps}

It is natural to use the inner product to provide insight into the various special classes of maps available in involutive categories.  To achieve this, it is useful to regard the inner product structure as providing an abstract metric on an object.  This then gives the notion of a metric-preserving map, also known as an isometry:

\begin{definition}
In an inner product category $h: X \to Y$ is an {\bf isometry} in case for every $f,g: A \to X$, $\langle f | g \rangle = \langle f h | g h \rangle$.
\end{definition}

It is immediate that isometries compose.

\begin{lemma} \label{isometry}
     The following are equivalent in an inner product category:
     \[ (i)~~ h: X \to Y \text{ is an isometry;} \quad \quad \quad \quad 
     (ii)~~ h \varphi_Y h^\dagger = \varphi_X \quad \quad \quad \quad 
     (iii)~~ hh^* = 1_X 
     \]
\end{lemma}

\begin{proof} ~
\begin{description}
\item{$(i) \Rightarrow  (ii)$:} 
If the inner product is preserved by $h$ then this preservation must hold for the identity maps on $X$ that is $\langle 1_X | 1_X \rangle = \langle h | h \rangle = h \langle 1_Y | 1_Y \rangle h^\dagger$ so that $h \varphi_Y h^\dagger = \varphi_X$ as required.
\item{$(ii) \Rightarrow (iii)$:}  If $h\varphi_Y h^\dagger = \varphi_X$
then $h h^* := h \varphi_Y h^\dagger \varphi_X^{-1} = \varphi_X \varphi_X^{-1} = 1_X.$
\item{$(iii) \Rightarrow (i)$:}
If $h h^* = 1_X$ then $\langle f h | g h \rangle = \langle f | gh h^* \rangle = \langle f | g \rangle$. \qedhere 
\end{description}
\end{proof}

Clearly if an isomorphism is an isometry its inverse is automatically an isometry:

\begin{lemma} \label{Lemma7.3}
In an inner product category, the following are equivalent for an isomorphism  $h: X \to Y$:
\begin{enumerate}[(i)]
\item $h$ and $h^*$ are isometries;
\item either $h$ or $h^{-1}$ is an isometry;
\item $h^\dagger =\varphi_Y^{-1} h^{-1} \varphi_X$
or equivalently $h = \varphi_X (h^{-1})^\dagger \varphi_Y^{-1}$;
\item $h^{-1} = h^*$.
\end{enumerate}
\end{lemma}
\begin{proof}
    \begin{description}
        \item[$(i) \Rightarrow (ii)$] Immediate.
        \item[$(ii) \Rightarrow (iii)$] Given that $h$ an isometry. Then, from Lemma \ref{isometry}(ii), we have that: $h \varphi_Y h^\dagger = \varphi_X$. Using the fact that $h$ is an isomorphism, we can deduce that $h^\dagger =\varphi_Y^{-1} h^{-1} \varphi_X$.   
        \item[$(iii) \Rightarrow (iv)$] Given that $h^\dagger = \varphi_Y^{-1} h^{-1} \varphi_X$. Then, from Lemma \ref{isometry}(iii), we have that $hh^* = 1_X$. Since, $h$ is an isomorphism, $hh^* = hh^{-1}$ which in turn implies that $h^* = h^{-1}$.
        \item[$(iv) \Rightarrow (i)$]
        When $h^{-1} = h^*$ then 
        \[ \langle f|g \rangle = \langle f hh^*|g \rangle
         = \langle f h|g h \rangle \]
         and so $h$ is an isometry.  Similarly, $h^*$ is an 
         isometry as 
         \[ \langle f h^*|g h^* \rangle 
            = \langle f |g h^* h^{**} \rangle
               = \langle f |g h^* h \rangle = \langle f |g \rangle. \qedhere \]
    \end{description}
\end{proof}
\begin{definition} 
In an inner product category an isomorphism is said to be {\bf unitary} if it satisfies any of the equivalent conditions of Lemma \ref{Lemma7.3}. 
\end{definition}

\begin{lemma} \label{Lemma 7.5} \label{unitary}~
\begin{enumerate}[(i)]
    \item $f: A \to X$ is a unitary map if and only if $f^\dagger: X^\dagger \to A^\dagger$ is unitary.
    \item The unitary structure map $\varphi_A$ is an isometry.
    \item The involutor map $\iota: A \to A^{\dagger \dagger}$ is an isometry.
\end{enumerate}
\end{lemma}
\begin{proof} ~
\begin{enumerate}[{\em (i)}]
\item First when $f$ is unitary, then $\varphi_X= f^{-1} \varphi_A (f^{\dagger})^{-1}$, so we have
\begin{eqnarray*}
    \varphi_{X^\dagger}=((f^{-1} \varphi_A (f^{\dagger})^{-1})^{-1})^{\dagger}
     = (f^\dagger \varphi_A^{-1} f)^{\dagger} =
    f^\dagger (\varphi_A^{-1})^\dagger f^{\dagger \dagger}
    = f^\dagger \varphi_{A^\dagger} f^{\dagger \dagger}
\end{eqnarray*}
So $f^\dagger$ is unitary. And when $f^\dagger$ is unitary , we have 
\begin{eqnarray*}
    (\varphi_{A}^{-1})^{\dagger} = (f^\dagger)^{-1} \varphi_{X^\dagger} (f^{\dagger \dagger})^{-1} = (f^\dagger)^{-1} (\varphi_{X}^{-1})^{\dagger} (f^{\dagger \dagger})^{-1} = ((f^\dagger)^{-1} \varphi_X^{-1} f^{-1})^{\dagger}
\end{eqnarray*}
So by faithfulness of $\dagger$ functor, we have $\varphi_A^{-1} =(f^\dagger)^{-1} \varphi_X^{-1} f^{-1}$, and hence $\varphi_A = f \varphi_X f^\dagger$, so $f$ is unitary.
\item For the unitary structure map, we have:
\begin{eqnarray*}
    \varphi_A \varphi_{A^\dagger} \varphi_{A}^{\dagger}= \varphi_A (\varphi_A^{-1})^{\dagger} \varphi_{A}^{\dagger} = \varphi_A (\varphi_A \varphi_A^{-1})^{\dagger} = \varphi_A
\end{eqnarray*}
So $\varphi_A: A \to A ^\dagger$ is  an isometry.
\item For the involutor we have, 
\begin{eqnarray*}
    \iota_A \varphi_{A^{\dagger\dagger}} \iota_A^{\dagger} = \iota_A (\varphi_{A^\dagger}^{-1})^{\dagger } \iota_A^{\dagger} = \iota_A \varphi_A^{\dagger} = \varphi_A
\end{eqnarray*}
where the last equality is by  Lemma \ref{unitary-identities} (ii). \qedhere
\end{enumerate}
\end{proof}
Note that the unitary structure map and the involutor are, of course, also unitary as they are isometries which are isomorphisms.

\medskip
\begin{definition}
A map $a: X \to X$ is {\bf Hermitian} in case, for every $f,g: A \to X$, 
$\langle f a| g \rangle = \langle f| g a \rangle$.
\end{definition}

\begin{lemma}
\label{Hermitian}
    In an inner product category the following are equivalent:
    \[ (i)~~ a: X \to X \text{ is Hermitian;} \quad \quad \quad \quad
    (ii)~~ a \varphi_X = \varphi_X a^\dagger; \quad \quad \quad \quad
    (iii)~~~ a=a^*. \]
\end{lemma}
\begin{proof} ~
\begin{description}
\item{$(i) \Rightarrow (ii)$:} If $a$ is Hermitian then $\langle 1_X a | 1_X \rangle = \langle 1_X | 1_X a \rangle$ so that $a \varphi_X = \varphi_X a^\dagger$.
\item{$(ii) \Rightarrow (iii)$:} By Lemma \ref{isometry}, $a^* = \varphi_X a^\dagger \varphi_X^{-1}$ so $a = a^*$ when $a \varphi_X = \varphi_X a^\dagger$.
\item{$(iii) \Rightarrow (i)$:}  Immediate from the definition of the adjoint. \qedhere
\end{description}
\end{proof}

\begin{definition} In an inner product category, 
\begin{itemize} 
\item A map $f: X \to X^\dagger$ is {\bf positive} if there is a map $q: X \to Y$ with 
$f = \langle q | q \rangle$. 
\item A map $f: X \to X$ is {\bf normal} if $\langle f|f\rangle \varphi_{X^\dagger} = \varphi_X \langle f^\dagger| f^\dagger\rangle$.
\end{itemize}
\end{definition}

\begin{remark}{\em How does this notion of positive maps relate to the usual notion of positive maps in a dagger category? Recall that, in a dagger category a map $f: X \rightarrow X$ is called positive if there exists a map $g: X \rightarrow Y$ such that $f= g g^\dagger$. When the involution is identity-on-objects. Our definition, in a $\dagger$-category, thus, reduces to this standard definition of positive maps. } 
\end{remark}

\begin{lemma} 
There are the following containment relations:
\begin{enumerate} [(i)]
        \item Every Hermitian map is normal.
        \item Every unitary endomorphism is normal.
    \end{enumerate}
\end{lemma}
For proof, see Appendix \ref{Sec: App special maps}.


\label{tab:special}

\bibliographystyle{plain}
\bibliography{main}

\newpage
\begin{appendix}

\section{Achiral involutions}
\label{Sec: App Achiral involutions}

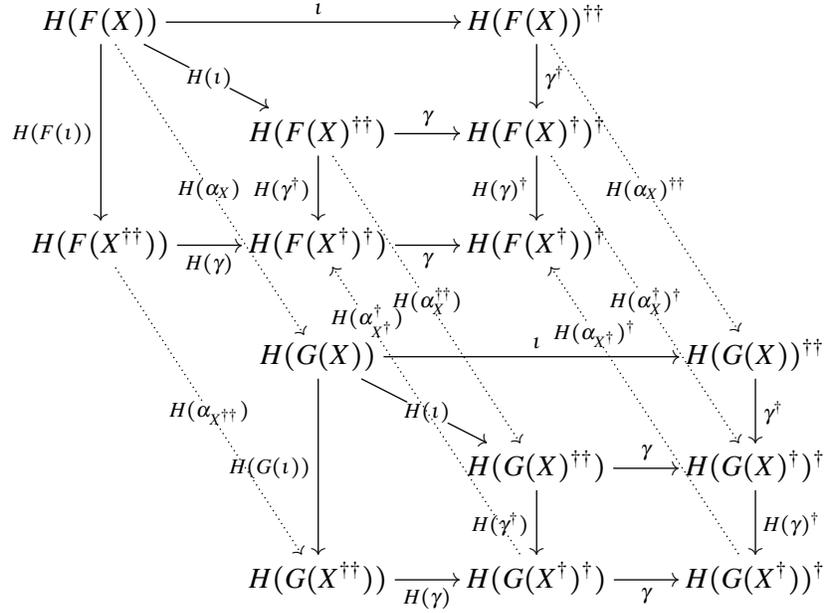
\begin{figure}[h]
    \centering
 \[ \xymatrix{ H(F(X)) \ar[rr]^\iota \ar[dd]_{H(F(\iota))} \ar[dr]|{H(\iota)} \ar@{..>}[dddr]|{H(\alpha_X)} 
                        & & H(F(X))^{\dagger\dagger} \ar[d]^{\gamma^\dagger} 
                             \ar@{..>}[dddr]|{H(\alpha_X)^{\dagger\dagger}} \\
                & H(F(X)^{\dagger\dagger}) \ar[r]^\gamma \ar[d]_{H(\gamma^\dagger)} \ar@{..>}[dddr]|{H(\alpha_X^{\dagger\dagger})}
                      & H(F(X)^\dagger)^\dagger \ar[d]_{H(\gamma)^\dagger} \ar@{..>}[dddr]|{H(\alpha_X^\dagger)^\dagger} \\
               H(F(X^{\dagger\dagger})) \ar[r]_{H(\gamma)} \ar@{..>}[dddr]|{H(\alpha_{X^{\dagger\dagger}})}
                      & H(F(X^\dagger)^\dagger) \ar[r]_{\gamma} 
                              \ar@{<..}[dddr]|<<<<<<<<<{H(\alpha_{X^\dagger}^\dagger)}
                      & H(F(X^\dagger))^\dagger \ar@{<..}[dddr]|<<<<<<<<<<<{H(\alpha_{X^\dagger})^\dagger} \\
                & H(G(X))   \ar[rr]^\iota \ar[dd]_{H(G(\iota))} \ar[dr]|{H(\iota)}  
                        & & H(G(X))^{\dagger\dagger} \ar[d]^{\gamma^\dagger} \\
                & & H(G(X)^{\dagger\dagger}) \ar[r]^\gamma \ar[d]_{H(\gamma^\dagger)} 
                      & H(G(X)^\dagger)^\dagger \ar[d]^{H(\gamma)^\dagger} \\
               & H(G(X^{\dagger\dagger})) \ar[r]_{H(\gamma)}
                      & H(G(X^\dagger)^\dagger) \ar[r]_{\gamma} 
                      & H(G(X^\dagger))^\dagger
               } \] 
    \caption{Whiskering on the right}
    \label{fig:Whisker}
\end{figure}

\FloatBarrier

\begin{lemma}
For any achiral involutive category the involutor itself is an (achiral) transformation.
\end{lemma}
\begin{proof}
Consider the following diagram:
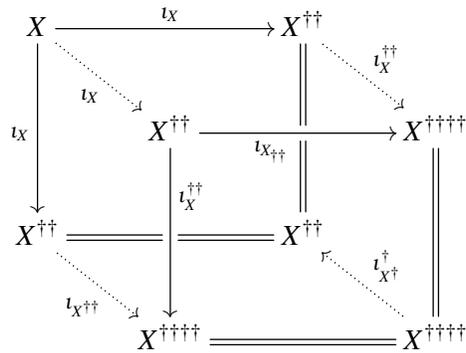
\[ \xymatrix{ X \ar[dd]_{\iota_X}\ar[rr]^{\iota_X} \ar@{..>}[dr]_{\iota_X} & & 
       X^{\dagger\dagger} \ar@{..>}[dr]^{\iota_X^{\dagger\dagger}} \ar@{=}[dd]|\hole  \\ 
       & X^{\dagger\dagger} \ar[rr]_<<<<<<<<<<{\iota_{X_{\dagger\dagger}}} \ar[dd]^<<<<<<<{\iota_X^{\dagger\dagger}}  &&   X^{\dagger\dagger\dagger\dagger} \ar@{=}[dd] \\
       X^{\dagger\dagger} \ar@{..>}[dr]_{\iota_{X^{\dagger\dagger}}} \ar@{=}[rr]|\hole && 
             X^{\dagger\dagger} \\
       & X^{\dagger\dagger\dagger\dagger} \ar@{=}[rr] && X^{\dagger\dagger\dagger\dagger} 
       \ar@{..>}[lu]_{\iota_{X^\dagger}^\dagger} } \]
  All the squares commute: notice that $\iota_X^{\dagger\dagger} = \iota_{X^{\dagger\dagger}}$ and 
  the bottom and rightmost square commute using the adjoint duality.
\end{proof}

\section{Inner product categories}
\label{Sec: App inner product}

\begin{definition}
\label{Defn: App inner product}
   An {\bf inner product category} is an involutive category which has an {\bf inner product combinator}:
   \[ \infer{\langle f|g \rangle: A \to B^\dagger}{f: A \to X & g: B \to X} \]
    satisfying axioms {\bf [IP.1]}--{\bf [IP.4]} listed below.

    \begin{enumerate}[{\bf [{IP}.1]}]
        \item {\bf Pre-composition:} $\langle hf | kg \rangle = h \langle f | g \rangle k ^\dagger$
        
        That is, given maps $f:A \to X$, $g: B \to X$, $h: Y \to A$, and $k: Z \to B$, the following diagram commutes:
        \[\begin{tikzcd}
        	Y & A \\
        	{Z^\dagger} & {B^\dagger}
        	\arrow["h", from=1-1, to=1-2]
        	\arrow["{\langle hf | kg \rangle}"', from=1-1, to=2-1]
        	\arrow["{=}"{description}, draw=none, from=1-1, to=2-2]
        	\arrow["{\langle f | g \rangle}", from=1-2, to=2-2]
        	\arrow["{k^\dagger}", from=2-2, to=2-1]
        \end{tikzcd}\]
        
        \item {\bf Dagger and duals:} (a) $\iota \langle f^\dagger | g^\dagger \rangle^\iota = \langle f | g \rangle $ and (b) $\iota \langle f | g \rangle^\dagger = \langle f | g \rangle$.
        
        That is, given $f: A \to X$ and $g: B \to X$, the following diagrams commute: 
        
        \[(a) \quad \quad \begin{tikzcd}
        	A \\
        	{A^{\dagger \dagger}} && {B^\dagger}
        	\arrow["\iota"', from=1-1, to=2-1]
        	\arrow["{\langle f | g \rangle}", from=1-1, to=2-3]
        	\arrow["{\langle f^\dagger | g^\dagger \rangle^\iota}"', from=2-1, to=2-3]
        \end{tikzcd} \quad \quad \quad 
        (b) \quad \quad \begin{tikzcd}
        	B \\
        	{B^{\dagger \dagger}} && {A^\dagger}
        	\arrow["\iota"', from=1-1, to=2-1]
        	\arrow["{\langle g | f \rangle}", from=1-1, to=2-3]
        	\arrow["{\langle f | g \rangle^\dagger}"', from=2-1, to=2-3]
        \end{tikzcd}\]
        
        \item Given $hg = 1$, $\langle f | g \rangle \langle h | k \rangle^\iota = fk$
        
        That is, given $f: A \to X$ and $g: B \to X$, and parallel maps $h,k: X \to B$, and that  $hg = 1_X$, the following diagram commutes: 
        \[\begin{tikzcd}
        	A & {} \\
        	{B^\dagger} && B
        	\arrow["{\langle f | g \rangle}"', from=1-1, to=2-1]
        	\arrow["fk", from=1-1, to=2-3]
        	\arrow["{\langle h | k \rangle^\iota}"', from=2-1, to=2-3]
        \end{tikzcd}\]
        
        \item Given $gh = 1$, $\langle f | g \rangle^\iota \langle h | k \rangle = (kf)^\dagger$.
        
        That is, given maps $f: A \to Z$, $g: A \to X$, and parallel maps $h,k: X \to A$, and that  $gh = 1_B$ then the following diagram commutes:
        \[\begin{tikzcd}
        	{Z^\dagger} \\
        	X && {X^\dagger}
        	\arrow["{\langle f | g \rangle ^\iota}"', from=1-1, to=2-1]
        	\arrow["{(kf)^\dagger}", from=1-1, to=2-3]
        	\arrow["{\langle h | k \rangle}"', from=2-1, to=2-3]
        \end{tikzcd}\]
    \end{enumerate} 
    Where the {\bf dual inner product} for maps $h: X \to A$ and $k: X \to B$, is the map $\langle h | k \rangle^\iota: A^\dagger \to B$ defined as follows:    
        \[ \text{\bf[dual-IP] } \quad \quad  
        \begin{tikzcd}
        	{A^\dagger} & {A^{\dagger \dagger \dagger}} \\
        	B & {B^{\dagger \dagger}}
        	\arrow["{\iota_{A^\dagger}}", from=1-1, to=1-2]
        	\arrow["{\langle h | k \rangle^\iota}"', from=1-1, to=2-1]
        	\arrow["{:=}"{description}, draw=none, from=1-1, to=2-2]
        	\arrow["{\langle k^\dagger | h^\dagger \rangle^\dagger}", from=1-2, to=2-2]
        	\arrow["{\iota_B^{-1}}", from=2-2, to=2-1]
        \end{tikzcd} \]
\end{definition}

\begin{lemma} 
The following hold in an inner product category:
\begin{enumerate}[(i)]
\item If $h,k: B \to A$ and $f,g: A \to X$ then $\langle hf|kg \rangle^\iota = f^\dagger \langle h|k \rangle^\iota g$
\item $\langle g^\dagger|f^\dagger\rangle = (\langle f|g \rangle^\iota)^\dagger$
\item $\langle f|g \rangle^\dagger = \langle g^\dagger | f^\dagger \rangle^\iota$
\end{enumerate}
\end{lemma}
\begin{proof} ~
    \begin{enumerate}[{\em (i)}]
        \item The definition of the dual inner product has to be carefully unwrapped while applying {\bf [IP.1]}:
        \begin{eqnarray*} \langle hf|kg \rangle^\iota
         & := & \iota \langle (kg)^\dagger| (hf)^\dagger  \rangle^\dagger \iota^{-1} 
        = \iota \langle g^\dagger k^\dagger| f^\dagger h^\dagger \rangle^\dagger \iota^{-1} \\
         & = & \iota (g^\dagger\langle  k^\dagger| h^\dagger \rangle f^{\dagger\dagger})^\dagger \iota^{-1} 
         =   \iota f^{\dagger\dagger\dagger} \langle  k^\dagger| h^\dagger \rangle^\dagger g^{\dagger\dagger}  \iota^{-1} \\
        & = & f^{\dagger} \iota \langle  k^\dagger| h^\dagger \rangle^\dagger  \iota^{-1} g  ~~=: f^{\dagger} \langle  h| k \rangle^\iota g
        \end{eqnarray*}
        \item We now show this identity holds using only the fact that the category is involutive:
    \[ \infer={\langle g^\dagger | f^\dagger \rangle = (\langle f|g \rangle^\iota)^\dagger 
         }{\infer={\langle g^\dagger | f^\dagger \rangle^\dagger =(\langle f|g \rangle^\iota)^{\dagger\dagger}
         }{\iota \langle g^\dagger | f^\dagger \rangle^\dagger 
         = \iota \langle g^\dagger | f^\dagger \rangle^\dagger \iota^{-1} \iota
         =: \langle f | g \rangle^\iota \iota = \iota ( \langle f|g\rangle ^\iota \rangle ^{\dagger\dagger}}} \]
         In particular, we use that $\iota$ is invertible and $(\_)^\dagger$ is a faithful functor.
        \item The two identities of {\bf [IP.2]} imply this identity 
        \[ \infer={\langle f^\dagger | g^\dagger \rangle^\iota = \langle g|f \rangle^\dagger}{\iota \langle f^\dagger | g^\dagger \rangle^\iota = \langle f|g \rangle
        = \iota \langle g|f \rangle^\dagger}\]
    \end{enumerate}
\end{proof}

\begin{lemma}
In an inner product category, the following are equivalent: 
\begin{enumerate}[(a)]
    \item Either {\bf [IP.2](a)} or {\bf [IP.2](b)} holds, and {\bf [IP.2](c)} $\langle f^{\dagger \dagger} | g^{\dagger \dagger} \rangle = \iota_A^{-1} \langle f | g \rangle \iota_{B^\dagger}$ holds;
    \item Both {\bf [IP.2](a)} and {\bf [IP.2](b)} hold.
\end{enumerate}
\end{lemma}
\begin{proof}
\begin{description}
    \item[$(\Rightarrow)$] 
    Suppose {\bf [IP.2](a)} holds, that is, $\iota_A \langle f^\dagger | g^\dagger \rangle^\iota = \langle f | g \rangle$, 
    and that the following equation holds 
    \[   \langle f^{\dagger \dagger} | g^{\dagger \dagger} \rangle = \iota_A^{-1} \langle f | g \rangle \iota_{B^\dagger}. \quad (\star) \]
    Then, we prove that {\bf [IP.2](b)} holds, that is, $\iota_B \langle f | g \rangle^\dagger = \langle g | f \rangle$:
    \begin{align*}
        \langle g | f \rangle 
             &\stackrel{\text{\bf [IP.2](a)}}{=} \iota_B \langle g^\dagger | f^\dagger \rangle^\iota 
            = \iota_B \iota_{B^{\dagger \dagger}} \langle f^{\dagger \dagger} | g^{\dagger \dagger} \rangle^\dagger \iota_{A^\dagger}^{-1} \\ 
            & \stackrel{\text{\bf [IP.2](c)}}{=} \iota_B \iota_{B^{\dagger \dagger}} \left(  \iota_A^{-1} \langle f | g \rangle \iota_{B^\dagger} \right)^\dagger \iota_{A^\dagger}^{-1} 
            \\&= \iota_B \iota_{B^{\dagger \dagger}} \iota_{B^\dagger}^\dagger \langle f | g \rangle^\dagger (\iota_A^{-1})^\dagger \iota_{A^\dagger}^{-1} \\ 
            & = \iota_B \iota_{B^{\dagger \dagger}} (\iota_{B^\dagger})^\dagger \langle f | g \rangle^\dagger \iota_{A^\dagger} \iota_{A^\dagger}^{-1} = \iota_B \langle f | g \rangle^\dagger
    \end{align*}
    Similarly, if {\bf [IP.2](b)} and $(\star)$ hold, then {\bf [IP.2](a)} holds.

    \item[$(\Leftarrow)$] 
    For the converse, assume {\bf [IP.2]}(a) and {\bf [IP.2]}(b) hold. Then, 
    \begin{align*}
        \iota_A^{-1} \langle f | g \rangle \iota_{B^\dagger} 
        & \stackrel{\text{\bf [IP.2]}(a)}{=} \iota_{A}^{-1} \iota_A \langle f^\dagger | g^\dagger \rangle^\iota \iota_{B^\dagger} 
        = \langle f^\dagger | g^\dagger \rangle^\iota \iota_{B^\dagger} \\ 
        &\stackrel{\text{\bf [dual-IP]}}{=} \iota_{A^{\dagger \dagger}} \langle g^{\dagger \dagger} | f^{\dagger \dagger} \rangle^\dagger \iota_{B^\dagger}^{-1} \iota_{B^\dagger} 
         \stackrel{\text{\bf [IP.2](b)}}{=} \iota_{A^{\dagger \dagger}} (\iota_{A^{\dagger \dagger}})^{-1} 
        \langle f^{\dagger \dagger} | g^{\dagger \dagger} \rangle \\
        & = \langle f^{\dagger \dagger} | g^{\dagger \dagger} \rangle
    \end{align*}
    \end{description}
\end{proof}

\begin{lemma} 
    Let $\mathbb{X}$ and $\mathbb{Y}$ be inner product categories. Then $(F, \gamma): \mathbb{X} \to \mathbb{Y}$, an involutive functor is inner product preserving if and only if it is a unitary functor.
\end{lemma}
\begin{proof}
\begin{description}
    \item[{\bf ($\Rightarrow$)}] 
        Let $(F, \gamma): \mathbb{X} \to \mathbb{Y}$ be a unitary functor, and $f, g: A \to X$ a pair of parallel maps in $\mathbb{X}$, then:
        \begin{eqnarray*}
             F (\langle f|g \rangle) \gamma_A &= &F( f \langle 1_X| 1_X \rangle g^\dagger) \gamma_A \quad \quad \quad \quad \quad \quad  \text{(by {\bf [IP.1]})} \\
             &= &F(f) F(\langle 1_X |1_X\rangle) F(g^\dagger) \gamma_A\\
             &=&F(f) F(\langle 1_X |1_X\rangle) \gamma_X (F(g))^\ddagger\\
             &=& F(f) F(\varphi_X) \gamma_X  (F(g))^\ddagger\\
             &=& F (f) \langle 1_{F(X)}| 1_{F(X)} \rangle (F(g))^\ddagger ~~~~~~~ \text{(as F is unitary)}\\
             &=& \langle F(f)| F(g) \rangle
        \end{eqnarray*}
        Thus, $F$ preserves the inner product.
    \item[{\bf($\Leftarrow$)}] 
        Suppose $F$ preserves inner product, then: \begin{eqnarray*}
        F(\varphi_A) \gamma_A = F(\langle 1_A| 1_A \rangle) \gamma_A = \langle F(1_A) |F(1_A) \rangle = \langle 1_{F(A)}| 1_{F(A)} \rangle = \varphi_{F(A)}
\end{eqnarray*}
\end{description}
\end{proof}

\subsection{Isometries and special maps}
\label{Sec: App special maps}

\begin{lemma} 
There are the following containment relations:
\begin{enumerate} [(i)]
        \item Every Hermitian map is normal.
        \item Every unitary endomorphism is normal.
    \end{enumerate}
\end{lemma}
\begin{proof}~
\begin{enumerate}[{\em (i)}]

\item For $a: X \to X$ which is Hermitian, we have,  $a \varphi_X= \varphi_X a^\dagger$, so
\[
    \langle a|a\rangle \varphi_{X^\dagger} = a \varphi_X a^\dagger \varphi_{X^\dagger}= aa \varphi_X \varphi_{X^\dagger}= aa \iota
\]
and 
\[
    \varphi_X \langle \varphi_X |a^\dagger\rangle = \varphi_X a^\dagger  \varphi_{X^\dagger} a^{\dagger \dagger}= a \varphi_X \varphi_{X^\dagger} a^{\dagger \dagger}= a\iota a^{\dagger \dagger} = aa \iota.
\]  
\item Let $a: X \to X$ be a unitary map, then first note by Lemma \ref{Lemma 7.5} (i), $a^\dagger: X^\dagger \to X^\dagger$ is also unitary. Then we have,
\[ \langle a|a\rangle \varphi_{X^\dagger}= a \varphi_X a^\dagger \varphi_{X^\dagger} = \varphi_{X} \varphi_{X^\dagger} = \iota \]
and
\[ \varphi_X \langle a^\dagger |a^\dagger\rangle = \varphi_X a^\dagger  \varphi_{X^\dagger} a^{\dagger \dagger} = \varphi_{X} \varphi_{X^\dagger} = \iota. \]
\end{enumerate}
\end{proof}

\end{appendix}

\end{document}